\documentclass[11pt]{amsart}
\usepackage[usenames]{color}
\usepackage{fullpage}
\usepackage{amscd}
\usepackage{amssymb, latexsym}
\usepackage[all,2cell,ps]{xy}
\usepackage{mathdots}

\theoremstyle{plain}
\newtheorem{thm}{Theorem}[section]

\newtheorem{cor}[thm]{Corollary}

\newtheorem{lm}[thm]{Lemma}

\theoremstyle{definition}
\newtheorem{de}[thm]{Definition}
\newtheorem{exm}[thm]{Example}

\newtheorem{remark}[thm]{Remark}

\newcommand{\wpf}{\mathcal{P}_{>0}^{<\om}}
\newcommand{\Om}{\Omega}
\newcommand{\om}{\omega}

\newcommand{\vv}{\mathcal{V}}
\newcommand{\vva}{\mathcal{V}}
\newcommand{\mm}{\mathcal{S}}

\newcommand{\typ}{\mho}

\def\sucdot{\operatornamewithlimits{\cup\!\!\!\cdot}}

\begin{document}

\title{Semilattice ordered algebras with constants}

\author{Agata Pilitowska}
\author{Anna Zamojska-Dzienio}
\address{(A.P., A.Z.) Faculty of Mathematics and Information Science, Warsaw University of Technology, Koszykowa 75, 00-662 Warsaw, Poland}
\email{(A.P.) A.Pilitowska@mini.pw.edu.pl}
\email{(A.Z.) A.Zamojska@mini.pw.edu.pl}
\date{\today}

\begin{abstract}
We continue our studies on semilattice ordered algebras. This time we accept constants in the type of algebras. We investigate identities satisfied by such algebras and describe the free objects in varieties of semilattice ordered algebras with constants.
\end{abstract}
\maketitle
\section{Introduction}\label{sec:Intr}

Ordered algebras such as Boolean algebras, Heyting algebras, lattice-ordered groups,
and MV-algebras have played a decisive role in logic, both as the models of theories of
first (or higher) order logic, and as algebraic semantics for the plethora of
non-classical logics emerging in the twentieth century from linguistics, philosophy,
mathematics, and computer science. For example, lattice-ordered groups play a
fundamental role in the study of algebras of logic, while MV-algebras are the algebraic
counterparts of the infinite-valued \L ukasiewicz propositional logic. 

Another important and widely investigated class is given by \emph{quantales} \cite{Ro90, Rump} - complete semilattices with an additional associative multiplication that distributes over
arbitrary joins. They were introduced in the 1980s as
a non-commutative generalization of \emph{locales}, to capture the non-commutative
logic arising in quantum mechanics. Quantales are examples of \emph{semilattice ordered algebras} (\emph{SLO algebras}, for short) we are interested in this paper.

In the series of papers \cite{PZ11}--\cite{PZ15} we have investigated SLO algebras $(A,\Om,\leq)$, where $(A,\leq)$ is a (join) semilattice, $(A,\Omega)$ is an algebra (where $\Omega$ is a set of operations of any finitary positive arity, and, moreover, $\Omega$ is not necessarily finite) and each operation from $\Omega$ distributes over the join. Obviously, examples are provided not only by already mentioned quantales or well known \emph{additively idempotent semirings} \cite{ChL18, G99, Z02}, but SLO algebras are much more general structures. The basic role in the theory was played by extended power algebras of non-empty subsets and extended algebras of (non-empty) subalgebras. The main aim of this paper is to describe the properties of SLO algebras with constants, i.e. we allow operations of the arity equal to zero (or in case of power constructions we allow the empty subset and the empty subalgebra). We study the relations between the SLO algebras with the signatures including and excluding constants. Our motivation was very natural and came from applications in logic, where constants $0$ and $1$ play a significant role. Similar research in case of \emph{commutative doubly-idempotent semirings} has been recently described in \cite{AJ, ChL18}.

The paper is organized as follows. In Section 2, we provide basic
definitions, results and examples concerning semilattice ordered algebras with and without various types of constants in the signature. In Section 3 we investigate identities satisfied by SLO algebras and we present a
necessary and sufficient condition for a SLO algebra to satisfy some
non-linear identity. In Section 4 we describe the free objects in an
arbitrary variety $\mm$ of semilattice ordered algebras (with various types of constants in the signature) and in the quasivariety of
$\Om$-subreducts of SLO algebras in $\mm$. In Section 5 we apply the results to some particular idempotent varieties of SLO algebras.

\section{SLO algebras}\label{sec:SLO}
Let $\typ$ be the variety of all algebras $(A,\Om)$ of a (fixed)
finitary type $\tau\colon\Om\to \mathbb{N}^+$ and let
$\vv\subseteq\typ$ be a subvariety of $\typ$. In \cite{PZ12a} we
introduced the following definition of a semilattice ordered
algebra.
\begin{de}\label{sec2:def1}
An algebra $(A,\Om,+)$ is called a {\bf\emph{semilattice ordered
$\vv$-algebra}} (or briefly {\bf\emph{semilattice ordered algebra}}) if
$(A,\Om)$ belongs to a variety $\vv$, $(A,+)$ is a (join)
semilattice (with semilattice order $\leq$, i.e. $x\leq y\;
\Leftrightarrow \; x+y=y$), and the operations from the set $\Om$
\emph{distribute} over the operation $+$, i.e. for each $n$-ary operation $\om\in \Om$, and $x_1,\ldots,x_i,y_i,\dots,x_n\in
A$
\begin{eqnarray*}
\om(x_1,\ldots,x_i+ y_i,\ldots,x_n)=\label{sec2:eq1}
\\
\om(x_1,\ldots,x_i,\ldots,x_n)+
\om(x_1,\ldots,y_i,\ldots,x_n),\nonumber
\end{eqnarray*}
for any $1\leq i \leq n$.
\end{de}

Definition \ref{sec2:def1} can be also formulated for semilattice ordered algebras with constants. Such constants may be of two types.
The first ones may be some special elements in the semilattice $(A,+)$ and the second ones may refer to the algebra $(A,\Om)\in \vv$. In particular, we can consider semilattice algebras with neutral element with respect to the operation $+$ or with unit elements with respect to operations in $\Om$.
\begin{de}
An algebra $(A,\Om,+,0)$ is called a $0$-{\bf\emph{semilattice ordered
$\vv$-algebra}} if $(A,\Om,+)$ is a semilattice ordered $\vv$-algebra, $(A,+,0)$ is the semilattice
with the least element $0$ and for each $\om\in \Om$ and $x_1,\ldots,x_i,\dots,x_n\in
A$
\[\om(x_1,\ldots,x_i,\ldots,x_n)=0
\]
whenever there is $1\leq i\leq n$ such that $x_i=0$.
\end{de}

\begin{de}
Let $A$ be a non-empty set. An element $\alpha\in A$ is called a {\bf\emph{unit for an $0\neq n$-ary operation}}
$\om\colon A^n\to A$, if for every $x\in A$
\[
\om(x,\alpha,\ldots,\alpha)=\om(\alpha,x,\alpha,\ldots,\alpha)=\ldots=\om(\alpha,\ldots,\alpha,x)= x.
\]
We say that $\alpha$ is a {\bf\emph{unit for an
algebra}} $(A,\Om)$, if it is a unit for each operation $\om\in \Om$.
\end{de}
\begin{remark}
Let $(A,\Om)$ be an algebra. Denote by $\mathcal{E}$ the set of all units for $(A,\Om)$. If there is a binary operation in $\Om$ then $\left|\mathcal{E}\right|\leq 1$, but in general $0\leq \left|\mathcal{E}\right|\leq \left|A\right|$. In this paper we assume that whenever a unit exists it is unique, i.e. we consider algebras $(A,\Om,\alpha)$ of a (fixed)
finitary type $\tau\colon\Om\cup\{\alpha\}\to \mathbb{N}$ with the unit $\alpha\in\mathcal{E}$. We denote this unique unit by $1$.
\end{remark}

Let $\typ_{1}$ be the variety of all algebras $(A,\Om,1)$ of a (fixed)
finitary type $\tau\colon\Om\cup\{1\}\to \mathbb{N}$ with the unit $1$ and such that $(A,\Om)\in\typ$ and let
$\vv_1\subseteq\typ_1$ be a subvariety of $\typ_1$.

\begin{de}
An algebra $(A,\Om,+,1)$  is a {\bf\emph{semilattice ordered $\vv_1$-algebra with a unit $1$}}
if $(A,\Om,+)$ is a semilattice ordered $\vv$-algebra and $1$ is a unit for $(A,\Om)$.
\end{de}
Note that we do not assume that $1$ is the greatest element in the semilattice $(A,+)$.
\begin{de}
An algebra $(A,\Om,+,0,1)$  is a $0$-{\bf\emph{semilattice ordered $\vv_1$-algebra with a unit}}
if $(A,\Om,+,0)$ is a $0$-semilattice ordered $\vv$-algebra and $1$ is a unit for $(A,\Om)$.
\end{de}

As a direct consequence of distributivity we immediately obtain that in a semilattice ordered algebra $(A,\Om,+)$
for each $n$-ary operation $\om\in \Om$ and
$x_{ij}\in A$ for $1\leq i\leq n$, $1\leq j\leq r$ we have
\begin{align}\label{sec2:eq3}
\om(x_{11},\ldots, x_{n1})+ \ldots +
\om(x_{1r},\ldots, x_{nr})\\
\leq \om(x_{11}+ \ldots + x_{1r},\ldots,x_{n1}+\ldots + x_{n
r}).\nonumber
\end{align}

It is also easy to notice that in semilattice ordered algebras all
$\Om$-operations are \emph{monotone} with respect to
the semilattice order $\leq$: if $x_i\leq y_i\in A$ for each $1\leq i\leq n$,
then
\begin{equation}\label{sec2:eq2}
\om(x_1,\ldots, x_n)\leq \om(y_1,\ldots, y_n).
\end{equation}

This means that such algebras form a
subclass of a class of ordered algebras in the sense of \cite{CL83}
(see also \cite{F66} and \cite{B76}). Basic examples are given by
additively idempotent semirings, distributive lattices, semilattice
ordered semigroups \cite{GPZ05}, semilattice ordered idempotent, entropic algebras (modals) \cite{PZ11}, extended power algebras \cite{PZ12a}
or semilattice modes \cite{K95}.

We start with providing a few natural examples of semilattice ordered algebras.

\begin{exm}{\bf Semilattice ordered semigroups.} An algebra $(A,\cdot,+)$, where $(A,\cdot)$ is a
semigroup, $(A,+)$ is a semilattice and for any $a,b,c\in A$,
$a\cdot(b+c)=a\cdot b+a\cdot c$ and $(a+b)\cdot c=a\cdot c+b\cdot c$
is a {\bf\emph{semilattice ordered semigroup}}. In particular,
semirings with an idempotent additive reduct (\cite{Z02},
\cite{PZ05}), {\bf\emph{dissemilattices}} (called also {\bf\emph{$\cdot$-distributive bisemilattices}} in \cite{MR79}) - algebras
$(M,\cdot,+)$ with two semilattice structures $(M,\cdot)$ and
$(M,+)$ in which the operation $\cdot$ distributes over the
operation $+$, and
distributive lattices are semilattice ordered
$\mathcal{SG}$-algebras, where $\mathcal{SG}$ denotes the variety of
all semigroups.

Another important class here is given by {\bf\emph{quantales}} \cite{Ro90, Rump}, i.e. semilattice ordered semigroups $(A,\cdot,+)$, where $(A,+)$ is a {\bf\emph{complete}} semilattice and the operation $\cdot$ distributes over arbitrary joins. 
\end{exm}

\begin{exm}{\bf Extended power
algebras of algebras.}\cite{PZ12a} For a given set $A$ denote by
$\mathcal{P}A$ the family of all subsets of $A$ and by $\mathcal{P}_{>0}A$
the family of all non-empty subsets of $A$. For
any $0\neq n$-ary operation $\om\colon A^n\to A$ we define {\bf\emph{the
complex operation}} $\om\colon(\mathcal{P}A)^n\to
\mathcal{P}A$ in the following way:
\begin{equation}\label{sec3:eq1}
\om(A_1,\dots,A_n):=\{\om(a_1, \dots, a_n)\mid a_i \in A_i\},
\end{equation}
where $\emptyset\neq A_1,\dots,A_n\subseteq A$
and
\begin{equation*}
\om(A_1,\dots,A_n):=\emptyset,
\end{equation*}
if there is $A_i=\emptyset$ for some $1\leq i\leq n$.

The set
$\mathcal{P}A$ also carries a join semilattice structure under
the set-theoretical union $\cup$. In \cite{JT51} B. J${\rm
\acute{o}}$nsson and A. Tarski proved that complex operations
distribute over the union $\cup$. Hence, for any algebra
$(A,\Om)\in\typ$, the {\bf\emph{extended power algebra}}
$(\mathcal{P}_{>0}A,\Om,\cup)$ is a semilattice ordered
$\typ$-algebra and the
$\emptyset$-{\bf\emph{extended power algebra}}
$(\mathcal{P}A,\Om,\cup,\emptyset)$ is a $0$-semilattice ordered
$\typ$-algebra.

\end{exm}
Notice that the algebra $(\mathcal{P}^{<\om}A,\Om,\cup,\emptyset)$ [$(\wpf A,\Om,\cup)$, respectively] of all finite
(non-empty) subsets of $A$ is a subalgebra of $(\mathcal{P}A,\Om,\cup,\emptyset)$
[$(\mathcal{P}_{>0}A,\Om,\cup)$, respectively].
Moreover the power algebra of all subsets of $A$ can also be viewed as
\emph{the Boolean algebra}
$(\mathcal{P}A,\cup,\cap,-,A,\emptyset,\Om)$ \emph{with operators}
$\Om$. This concept was introduced and studied by B. J\'{o}nsson and
A. Tarski \cite{JT51, JT52}.

\begin{exm}{\bf Extended power
algebras of algebras with a unit.}
Let $(A,\Om,1)$ be an algebra with the unit $1\in A$. It is clear that for any non-empty subset $X\subseteq A$ and $n$-ary operation $\om\in \Om$
\[
\om(\{1\},\ldots,\underbrace{X}_{i},\ldots,\{1\})=\{\om(1, \dots, x_i,\ldots,1)\mid x_i \in X\}=\{x_i\mid x_i \in X\}=X.
\]
Then the algebra
$(\mathcal{P}_{>0}A,\Om,\cup,\{1\})$ is a semilattice ordered
algebra with the unit $\{1\}$ and the algebra
$(\mathcal{P}A,\Om,\cup,\emptyset,\{1\})$ is a $0$-semilattice ordered
algebra with the unit $\{1\}$.
\end{exm}

For a semigroup $(A,\cdot)$ [a monoid $(A,\cdot,1)$, respectively] its $\emptyset$-extended power algebra $(\mathcal{P}A,\cdot,\cup,\emptyset)$ [$\emptyset$-extended power algebra with a unit $(\mathcal{P}A,\cdot,\cup,\emptyset,\{1\})$, respectively] is a basic example of a quantale [unital quantale, respectively].

An algebra $(A,\Om)$ is {\bf\emph{idempotent}} if each
singleton is a subalgebra, i.e. for every $n$-ary operation $\om\in \Om$ and $x\in A$ the following is
satisfied:
\begin{eqnarray*}
\om(x,\ldots,x)= x. 
\end{eqnarray*}
A variety $\vva$ of algebras is called \emph{idempotent} if every
algebra in $\vva$ is idempotent.

An algebra $(A,\Om)$ is  {\bf\emph{entropic}} if any two of its operations
commute. This property may also be expressed by means
of identities: for every $m$-ary $\om\in \Om$ and $n$-ary $\varphi\in \Om$ operations and $x_{11},\dots,x_{n1},\dots,x_{1m},\dots,x_{nm}\in A$
\begin{eqnarray*}
\om(\varphi(x_{11},\dots,x_{n1}),\dots,\varphi(x_{1m},\dots,x_{nm}))
=\\
\varphi(\om(x_{11},\dots,x_{1m}),\dots,\om(x_{n1},\dots,x_{nm})). 
\end{eqnarray*}
\begin{remark}
If there is a unit $1$ in an entropic algebra $(A,\Om)$ then the algebra is {\bf\emph{symmetric}}, i.e. for every $n$-ary operation $\om\in\Om$ and $x_1,\dots,x_n\in A$ the following identity holds:
\begin{eqnarray*}
\om(x_1,\ldots,x_n)=\om(x_{\pi(1)},\ldots,x_{\pi(n)}),
\end{eqnarray*}
for each permutation $\pi$  
of the set $\{1,\ldots ,n\}$.
\end{remark}

\begin{exm}{\bf Modals.}
A {\bf\emph{modal}} $(M,\Om,+)$ is a semilattice ordered algebra in which the algebra $(M,\Om)$ is idempotent and entropic.
  Examples of modals include
semilattice ordered semilattices (dissemilattices) and the algebra $(\mathbb{R},\underline{I}^0,max)$
defined on the set of real numbers, where $\underline{I}^0$ is the
set of the binary operations:
\[
\underline{p}:\mathbb{R}\times \mathbb{R}\rightarrow \mathbb{R}, \;
(x,y)\mapsto (1-p)x+py,
\] for each $p\in (0,1)\subset \mathbb{R}$.

Idempotent and entropic algebras (called {\bf\emph{modes}}) and also modals were introduced and investigated in detail by A.
Romanowska and J.D.H. Smith (\cite{RS85}-\cite{RS02}). In particular, they showed that for a given idempotent and entropic algebra $(M,\Om)$, the sets
$S_{>0}(M)$ of non-empty subalgebras and $P_{>0}(M)$ of finitely generated non-empty subalgebras under the complex operations $\om\in\Om$ and ordered by set-theoretic inclusion are modals. In case, we allow empty subalgebras, one obtains $0$-semilattice ordered modes: $(S(M),\Om,\cup,\emptyset)$ and $(P(M),\Om,\cup,\emptyset)$.

If a modal $(M,\Om,+)$ is entropic, then it is an example of a {\bf\emph{semilattice mode}}. Semilattice modes
were described by K. Kearnes in \cite{K95}.
\end{exm}

\section{Identities in SLO algebras}\label{sec:id}

As we will see in Section \ref{sec:free} extended power algebras of algebras and their $\Om$-reducts play a special role in the context of semilattice ordered algebras. Let us recall some fundamental results referring to such algebras.

Let $\vva\subseteq\typ$ be a variety of algebras and let
\[
\vva\Sigma:={\rm HSP}(\{(\mathcal{P}A,\Om)\mid (A,\Om)\in\vva\}), \; {\rm and}
\]
\[
\vva\Sigma_{>0}:={\rm HSP}(\{(\mathcal{P}_{>0}A,\Om)\mid (A,\Om)\in\vva\}).
\] Let us consider their subvarieties

\[
\vva\Sigma^{< \om}:={\rm HSP}(\{(\mathcal{P}^{< \om} A,\Om)\mid
(A,\Om)\in\vva\}), \; {\rm and}
\]
\[
\vva\Sigma^{< \om}_{>0}:={\rm HSP}(\{(\mathcal{P}_{>0}^{< \om}A,\Om)\mid
(A,\Om)\in\vva\})\]  of power algebras of finite subsets.

We call a term $t$ of the language of a variety $\vv$ {\bf\emph{linear}},
if every variable occurs in $t$ at most once. An identity $t\approx u$ is
called {\bf\emph{linear}}, if both terms $t$ and $u$ are linear.

Note
that the definition \eqref{sec3:eq1} of a complex operation extends
to each linear derived operation $t$:
\begin{equation}\label{sec1:eq1}
t(A_1,\dots,A_n):=\{t(a_1, \dots, a_n)\mid a_i \in A_i\}.
\end{equation}

Each non-linear term $t$ can be obtained from a linear one $t^*$ by
identification of some variables. Let
$t^*(x_{11},\ldots,x_{1k_1},\ldots,x_{m1},\ldots,x_{mk_m})$ be a
linear term such that
\begin{equation*}
t(x_1,\ldots,x_m)=t^*(\underbrace{x_{1},\ldots,x_{1}}_{k_1-times},\ldots,\underbrace{x_{m},\ldots,x_{m}}_{k_m-times}).
\end{equation*}
Then for any subsets $A_1,\dots,A_m$
\begin{align*}
&\{t(a_1, \dots, a_m)\mid a_i \in A_i\}\subseteq t(A_1,\dots,A_m)=\\
&\{t^*(a_{11}, \dots, a_{1k_1},\ldots,a_{m1},\ldots,a_{mk_m})\mid a_{ij} \in A_i\}=\\
&t^*(\underbrace{A_1,\ldots,A_1,}_{k_1-times}\ldots,\underbrace{A_m,\ldots,A_m}_{k_m-times}).
\end{align*}

G. Gr\"atzer and H. Lakser proved in \cite{GL88} that for any subvariety
$\vv\subseteq\mho$ the following result holds.
\begin{thm}\cite[Theorem 1]{GL88}\label{gl}
Let $\vva$ be a variety of algebras. The variety $\vva\Sigma_{>0}$
satisfies precisely those identities resulting through
identification of variables from the linear identities true in
$\vva$.
\end{thm}

\begin{cor}\cite{PZ12b}\label{gl1}
Let $\vva$ be a variety of algebras. The varieties $\vva\Sigma_{>0}$ and
$\vva\Sigma^{< \om}_{>0}$ coincide.
\end{cor}

An identity $t\approx u$ is called {\bf\emph{regular}}, if the set of variable symbols occurring in $t$ equals the set of variable symbols occurring in $u$.
The following results are analogs of above ones formulated for varieties of power algebras including the empty set.
\begin{thm}\cite[Theorem 2]{GL88}\label{gl2}
Let $\vva$ be a variety of algebras. The variety $\vva\Sigma$
satisfies precisely those regular identities resulting through
identification of variables from the linear identities true in
$\vva$.
\end{thm}

\begin{cor}
Let $\vva$ be a variety of algebras. The varieties $\vva\Sigma$ and
$\vva\Sigma^{< \om}$ coincide.
\end{cor}
\begin{cor}\cite[Corollary 2]{GL88},\cite[Theorem 4.6]{PZ15}\label{gl3}
Let $\vva$ be a variety of algebras. Then $\vva=\vva\Sigma^{< \om}_{>0}$ if and only if $\vva$ is defined by a set of linear identities, and $\vva=\vva\Sigma^{< \om}$ if and only if $\vva$ is defined by a set of linear regular identities.
\end{cor}

Let $(A,\Gamma)$ be an algebra of a given type $\tau:\Gamma
\rightarrow \mathbb{N}$. Denote by $\mathfrak{B}\Gamma$ a set of
derived (or term) operations of $\Gamma$ and let $\Om\subseteq
\mathfrak{B}\Gamma$. An algebra $(A,\Om)$ is said to be a
\emph{reduct} (\emph{$\Om$-reduct}) of the algebra $(A,\Gamma)$. A
subalgebra of a reduct of $(A,\Gamma)$ is called a \emph{subreduct}.

Let $(A,\Om,+)$ be a semilattice ordered algebra generated by a set $X\subseteq A$. Denote by
$(\langle X\rangle_{\Om},\Om)$ the subalgebra of the $\Om$-reduct
$(A,\Om)$ generated by the set $X$. The algebra $(\langle
X\rangle_{\Om},\Om)$ 
contains all elements from $(A,\Om,+)$ obtained as results of \emph{derived} (or \emph{term})
operations from $\Om$ on the set $X$. We will call it
the \emph{full $\Om$-algebra subreduct $($of a semilattice ordered
algebra $(A,\Om,+)$$)$ relative to $X$}.

An element $r\in A$  is said to be in {\bf\emph{disjunctive form}} if it is
a join of a finite number of elements from their full $\Om$-subreduct $\langle X\rangle_{\Om}$.

The following theorem shows that each element in a semilattice ordered algebra may be
expressed in such form.

\begin{lm}{\rm (Disjunctive Form Lemma).}\label{slowa}
Let $(A,\Om,+)$ be a semilattice ordered algebra generated by a set $X\subseteq A$. For
each $r\in A$, there exist $r_1,\ldots,r_p\in \langle
X\rangle_{\Om}$ such that
\[r=r_1+\ldots+r_p.
\]
\end{lm}
\begin{proof}
The proof goes by induction on the minimal number $m$ of occurrences
of the semilattice operation $+$ in the expression of $r$ as a semilattice ordered algebra
word in the alphabet $X$.

Consider $r=r_1$ with $r_1\in \langle X\rangle_{\Om}$. Hence, the
result holds for $m=0$.

Now suppose that the hypothesis is established for $m>0$ and let
$r\in A$ be an element in which the semilattice operation $+$ occurs
$m+1$ times. Let $r=r_1+r_2$, for some $r_1,r_2\in A$. By
induction hypothesis there are $r_{11},\ldots,r_{1k},
r_{21},\ldots,r_{2n}\in \langle X\rangle_{\Om}$ such that
$$r=r_1+r_2=r_{11}+\ldots+r_{1k}+r_{21}+\ldots+r_{2n}.$$
Otherwise, $r=\om(r_1,\ldots, r_k+s_k,\ldots,r_n)$
for some $\om\in \Om$ and $r_1,\ldots, r_k, \ldots,r_n,s_k\in A$. Then, by
distributivity we have
$$r=\om(r_1,\ldots, r_k+s_k,\ldots,r_n)=\om(r_1,\ldots,r_k,\ldots,
r_n)+\om(r_1,\ldots,s_k,\ldots, r_n).$$ Because
$\om(r_1,\ldots,r_k,\ldots, r_n),\om(r_1,\ldots,s_k,\ldots, r_n)\in
A$, this completes the inductive proof.
\end{proof}

\begin{cor}\label{l6} Let $(A,\Om,+)$ be a semilattice ordered algebra generated by a
set $X\subseteq A$. There is a set $Y\subseteq A$ of generators of
the semilattice $(A,+)$ such that $Y\subseteq \langle
X\rangle_{\Om}$.
\end{cor}

Let $\mathcal{SV}$ denote the class of all semilattice ordered
algebras such that for each $(A,\Om,+)\in\mathcal{SV}$ there exists
a set of generators such that their full $\Om$-subreduct 
lies in $\vv$.

\begin{thm}\label{linrow} Let $\vv$ be a variety of $\Om$-algebras
satisfying an identity $t\approx u$ for some $0\neq n$-ary terms and let $\mm\subseteq \mm\vv$  be a
variety of semilattice ordered algebras $(A,\Om,+)$ such that the word operation
$t:A^n\rightarrow A$ distributes over the operation $+$.

Then the identity $t\approx u$ is satisfied in $\mm$ if and only if the
word operation $u:A^n\rightarrow A$ distributes over the operation
$+$.
\end{thm}
\begin{proof} Let $(A,\Om,+)\in \mm\subseteq \mm\vv$ and let the word operation $t:A^n\rightarrow A$ distribute over the operation $+$.

Because the variety $\mm$ is, by assumption, included in $\mm\vv$,
there exists a set $X$ of  generators of $(A,\Om,+)$, such that its
full $\Om$-algebra subreduct relative to $X$ belongs to the variety
$\vv$. Hence, the identity $t\approx u$ is also true in $(\langle
X\rangle_{\Om},\Om)$.

Suppose first that the word operation $u:A^n\rightarrow A$
distribute over the operation $+$ and let $r_1,\ldots,r_n\in A$. By
the Disjunction Form Lemma \ref{slowa} there exist
$r_{11},\ldots,r_{1k_1},\dots,r_{n1},\ldots,r_{nk_n}\in \langle
X\rangle_{\Om}$ such that for each $1\leq i \leq n$,
$r_i=r_{i1}+\ldots+r_{ik_i}$. Then, by distributivity of
operations $t:A^n\rightarrow A$ and $u:A^n\rightarrow A$ we obtain
\begin{align*}
&t(r_1,\ldots,r_n)=t(r_{11}+\ldots+r_{1k_1},\dots,r_{n1}+\ldots+r_{nk_n})=\\
&\sum\limits_{1\leq i \leq n \atop a_i\in
\{r_{i1},\ldots,r_{ik_i}\}}t(a_1,\ldots,a_n)=\sum\limits_{1\leq i
\leq n \atop a_i\in
\{r_{i1},\ldots,r_{ik_i}\}}u(a_1,\ldots,a_n)=\\
&u(r_{11}+\ldots+r_{1k_1},\dots,r_{n1}+\ldots+r_{nk_n})=u(r_1,\ldots,r_n).
\end{align*}

The converse implication is obvious.
\end{proof}

\begin{exm}
Let $(A,\Om,+)$ be a semilattice ordered algebra and let $\om\in\Om$ be an $n$-ary operation.
The unary operation $t(x):=\om(x,\ldots,x)\colon A\rightarrow A$ distributes over
the operation $+$ if and only if for any $x,y\in A$
\begin{align*}
&t(x)+t(y)=\sum_{x_i\in \{x,y\}}\om(x_1,\ldots,x_n).
\end{align*}
In particular, if $\om\in \Om$ is a binary idempotent operation, then  the operation $t(x)=\om(x,x)\colon A\rightarrow A$ distributes over
the operation $+$ if and only if for any $x,y\in A$
\begin{align*}
&x+y=x+y+\om(x,y)+\om(y,x).
\end{align*}
\end{exm}
\begin{lm}\label{ll} Let $(A,\Om,+)$ be a semilattice ordered algebra and let $t$ be an $0\neq n$-ary linear
$\Om$-term. The word operation $t:A^n\rightarrow A$ distributes over
the operation $+$.
\end{lm}

\begin{proof}
The proof will go by induction on the minimal number $m$ of
occurrences of (symbols of) the basic $\Om$-operations in the
corresponding linear $\Om$-term.

By definition of a semilattice ordered algebra, the lemma is certainly true for $m=1$. Now
suppose that the hypothesis is established for $m>1$. Let
$$t(x_{11},\ldots,x_{kp_{k}})=\om(\nu_1(x_{11},\ldots,x_{1p_{1}}),\ldots,\nu_k(x_{k1},\ldots,x_{kp_{k}}))$$
be a linear $\Om$-term, for some $\om\in \Om$, different variable
symbols $x_{11},\ldots,x_{1p_{1}},\ldots,x_{k1},\ldots,x_{kp_{k}}$
and linear $\Om$-words $\nu_1,\dots,\nu_k$, in which the basic
$\Om$-operations occur $m+1$ times.

By induction hypothesis, the $\Om$-word operations $\nu_i:A^{p_i}\rightarrow A$, for $1\leq i \leq k$, distribute over the operation $+$. This implies that for any $x_{11},\ldots,x_{1p_{1}},\ldots,x_{i1},\ldots,x_{ij},y_{ij},\ldots,x_{ip_i},\\ \ldots,x_{k1},\ldots,x_{kp_{k}}\in A$,
\begin{align*}
&t(x_{11},\ldots,x_{ij}+y_{ij},\ldots,x_{kp_{k}})=\\
&\om(\nu_1(x_{11},\ldots,x_{1p_{1}}),\ldots,\nu_i(x_{i1},\ldots,x_{ij}+y_{ij},\ldots,x_{ip_{i}}),\ldots,\nu_k(x_{k1},\ldots,
x_{kp_{k}}))=\\
&\om(\nu_1(x_{11},\ldots,x_{1p_{1}}),\ldots,\nu_i(x_{i1},\ldots,x_{ij},\ldots,x_{ip_{i}})+ \\
&\nu_i(x_{i1},\ldots,y_{ij},\ldots, x_{ip_{i}}),\ldots, \nu_k(x_{k1},\ldots,x_{kp_{k}}))=\\
&\om(\nu_1(x_{11},\ldots,x_{1p_{1}}),\ldots,\nu_i(x_{i1},\ldots,x_{ij},\ldots,x_{ip_{i}}),\ldots,\nu_k(x_{k1},\ldots,x_{kp_{k}}))+ \\
&\om(\nu_1(x_{11},\ldots,x_{1p_{1}}),\ldots,\nu_i(x_{i1},\ldots,y_{ij},\ldots, x_{ip_{i}}),\ldots, \nu_k(x_{k1},\ldots,x_{kp_{k}}))=\\
&t(x_{11},\ldots,x_{ij},\ldots,x_{kp_{k}})+t(x_{11},\ldots,y_{ij},\ldots,x_{kp_{k}}),
\end{align*}
which finishes the proof.
\end{proof}

\begin{cor}\label{c5} Let $\vv$ be a variety of $\Om$-algebras satisfying
an identity $t\approx u$ for some $0\neq n$-ary terms, where $t$ is linear. The identity $t\approx u$ is true
in a variety $\mm\subseteq\mm\vv$ of semilattice ordered algebras if and only if the word
operation $u:A^n\rightarrow A$ distributes over the operation $+$.
\end{cor}

\begin{cor}
A variety $\mm\subseteq\mm\vv$ of semilattice ordered algebras satisfies each linear
identity true in $\vv$.
\end{cor}

\begin{cor}\label{c6} Let $\vv$ be a variety of $\Om$-algebras satisfying
an identity $\om(x,\ldots,x)=x$, for $\om\in \Om$. The identity $\om(x,\ldots,x)=x$ is true
in a variety $\mm\subseteq\mm\vv$ of semilattice ordered algebras if and only if
the following identity
\begin{align*}
&x+y=\sum_{x_i\in \{x,y\}}\om(x_1,\ldots,x_n)
\end{align*}
is true in $\mm$.
\end{cor}

Let $^{\cup}\vv\Sigma^{< \om}_{>0}$ denote the variety of semilattice ordered algebras generated by
extended power algebras of finite non-empty subsets of algebras from $\vv$, i.e.,
$$^{\cup}\vv\Sigma^{< \om}_{>0}:={\rm HSP}(\{(P_{>0}^{<\om}A,\Om,\cup)\mid (A,\Om)\in \vv\}).$$

\begin{thm}\label{thm13}
Let $\vv$ be a variety of $\Om$-algebras. The variety
$\vv\Sigma^{< \om}_{>0}$ is locally finite if and only if the variety
$^{\cup}\vv\Sigma^{< \om}_{>0}$ is locally finite.
\end{thm}

\begin{proof}
Let $(C,\Om,\cup)\in ^{\cup}\vv\Sigma^{< \om}_{>0}$ be an algebra
generated by a finite set $X\subseteq C$. By Disjunctive Form Lemma
\ref{slowa}, for each $a\in C$, there exist $a_1,\ldots,a_p\in
\langle X\rangle_\Om$ such that
\begin{equation}\label{r4}
a=a_1\cup\ldots\cup a_p.
\end{equation}
If the variety $\vv\Sigma^{<\omega}_{>0}$ is locally finite, then the
algebra $\langle X\rangle_\Om \in \vv\Sigma^{<\omega}_{>0}$ is finite. Hence,
there are only finitely many elements of the form (\ref{r4}).
Consequently, the algebra $(C,\Om,\cup)$ is finite.

Let $(F_{^{\cup}\vv\Sigma^{<\omega}_{>0}}(X),\Om,\cup)$ be the free
algebra in the variety $^{\cup}\vv\Sigma^{< \om}_{>0}$ generated by a set $X$. It is known
that free algebra over $X$ in the variety generated by
$\Om$-subreducts of algebras in $^{\cup}\vv\Sigma^{< \om}_{>0}$ is
isomorphic to the $\Om$-subreduct $(\langle X\rangle,\Om)$,
generated by X, of the free algebra
$(F_{^{\cup}\vv\Sigma^{< \om}_{>0}}(X),\Om,\cup)$. (See e.g.
\cite[Theorem 3.9]{PZ11}). The free algebra
$(F_{\vv\Sigma^{<\omega}_{>0}}(X),\Om)$ is then a homomorphic image of
$(\langle X\rangle_\Om,\Om)$. Consequently, the variety
$\vv\Sigma^{<\omega}_{>0}$ is locally finite if the variety
$^{\cup}\vv\Sigma^{< \om}_{>0}$ is locally finite.
\end{proof}

Note that the same is true also for varieties generated by power algebras of all finite subsets (i.e. including the empty set).
\section{Free SLO algebras with constants}\label{sec:free}

Let $(F_{\vv}(X),\Om)$ be the free algebra over a set $X$ in
the variety $\vv\subseteq\typ$ and $(F_{\vv_1}(X),\Om)$ be the free algebra over a set $X$ in
the variety $\vv_1\subseteq\typ_1$.  Let $\mm_{\vv}$ denote the variety of all
semilattice ordered $\vv$-algebras, $\mm_{\vv}^{0}$ denote the variety of all
$0$-semilattice ordered $\vv$-algebras, $\mm_{\vv_1}$ denote the variety of all
semilattice ordered $\vv_1$-algebras and $\mm_{\vv_1}^0$ denote the variety of all
$0$-semilattice ordered $\vv_1$-algebras.

\begin{thm}[Universality Property for Semilattice Ordered Algebras]\label{thm:up1}\cite{PZ12a} 
Let $X$ be an arbitrary set and $(A,\Om,+)\in \mm_{\vv}$. Each mapping $h\colon X\to A$ can
be extended to a unique homomorphism
$\overline{\overline{h}}\colon(\wpf F_{\vv}(X),\Om,\cup)\to (A,\Om,+)$, such that
$\overline{\overline{h}}|_{X}=h$.
\end{thm}

\begin{cor}[Universality Property for $0$-Semilattice Ordered Algebras] \label{cor:up2} 
Let $X$ be an arbitrary set and $(A,\Om,+,0)\in \mm_{\vv}^{0}$. Each mapping $h\colon X\to A$ can
be extended to a unique homomorphism
$\overline{\overline{h}}\colon(\mathcal{P}^{<\om}F_{\vv}(X),\Om,\cup,\emptyset)\to (A,\Om,+,0)$, such that
$\overline{\overline{h}}|_{X}=h$.
\end{cor}

\begin{proof}
Let $(A,\Om,+,0)\in\mm_{\vv}^{0}$. Since
$(A,\Om)\in \vv$ then any mapping $h\colon X\to A$
may be uniquely extended to an $\Om$-homomorphism
$\overline{h}\colon(F_{\vv}(X),\Om)\to (A,\Om)$.

Let us define the mapping $\overline{\overline{h}}\colon(\mathcal{P}^{<\om} F_{\vv}(X),\Om,\cup,\emptyset)\to (A,\Om,+,0)$ by
\[
\overline{\overline{h}}(T)=\sum\limits_{t\in
T}\overline{h}(t),
\]
if $T$ is a non-empty finite subset of $F_{\vv}(X)$, and
\[
\overline{\overline{h}}(\emptyset)=0.
\]

By Theorem \ref{thm:up1} the mapping $\overline{\overline{h}}|_{\mathcal{P}^{<\om}_{>0} F_{\vv}(X)}$ is the unique $\{\Om,\cup\}$-homomorphism  such that
$\overline{\overline{h}}|_{X}=h$. But obviously if for some $T_i=\emptyset$, we also have:
\[
\overline{\overline{h}}(\om(T_1,\ldots,\emptyset,\ldots,T_n))=\overline{\overline{h}}(\emptyset)=0
=\om(\overline{\overline{h}}(T_1),\ldots,0,\ldots,\overline{\overline{h}}(T_n))=
\om(\overline{\overline{h}}(T_1),\ldots,\overline{\overline{h}}(\emptyset),\ldots,\overline{\overline{h}}(T_n)).
\]
Moreover, for $T_1=\emptyset$
\[
\overline{\overline{h}}(\emptyset\cup T_2)=\overline{\overline{h}}(T_2)=0+\overline{\overline{h}}(T_2)=\overline{\overline{h}}(\emptyset)+ \overline{\overline{h}}(T_2),
\]
which shows that the mapping $\overline{\overline{h}}|_{\mathcal{P}^{<\om} F_{\vv}(X)}$ is an $\{\Om,\cup,\emptyset\}$-homomorphism, and finishes the proof.
\end{proof}

\begin{cor} [Universality Property for Semilattice Ordered Algebras with a unit] \label{cor:up3}
Let $X$ be an arbitrary set and $(A,\Om,+,1)\in \mm_{\vv_1}$. Each mapping $h\colon X\to A$ can
be extended to a unique homomorphism
$\overline{\overline{h}}\colon(\wpf F_{\vv_1}(X),\Om,\cup,\{1\})\to (A,\Om,+,1)$, such that
$\overline{\overline{h}}|_{X}=h$.
\end{cor}

\begin{proof}
Let $(A,\Om,1)$ be a $\vv_1$-algebra with the unit $1\in A$.
Then for the $\{\Om,1\}$-homomorphism
$\overline{h}\colon(F_{\vv_1}(X),\Om,1)\to (A,\Om,1)$, which extends a mapping $h\colon X\to A$ one has that $\overline{h}(1)=1$.
Further, by Theorem \ref{thm:up1}, the mapping
$\overline{\overline{h}}\colon(\wpf F_{\vv_1}(X),\Om,\cup)\to (A,\Om,+)$,
\[
\overline{\overline{h}}(T)=\sum\limits_{t\in
T}\overline{h}(t)
\] for a non-empty finite subset $T$  of $F_{\vv_1}(X)$, is a homomorphism such that
$\overline{\overline{h}}|_{X}=h$. In particular, for $T=\{1\}$
\[
\overline{\overline{h}}(\{1\})=\overline{h}(1)=1.
\]
\end{proof}

As a direct corollary from Corollaries \ref{cor:up2}-\ref{cor:up3} we obtain
\begin{cor} [Universality Property for $0$-Semilattice Ordered Algebras with a unit]\label{cor:up4}
Let $X$ be an arbitrary set and $(A,\Om,+,0,1)\in \mm_{\vv_1}^{0}$. Each mapping $h\colon X\to A$ can
be extended to a unique homomorphism
$\overline{\overline{h}}\colon(\mathcal{P}^{<\om} F_{\vv_1}(X),\Om,\cup,\emptyset,\{1\})\to (A,\Om,+,0,1)$, such that
$\overline{\overline{h}}|_{X}=h$.
\end{cor}
By Theorem \ref{thm:up1} and Corollaries \ref{cor:up2}-\ref{cor:up4}, for an arbitrary variety
$\vv\subseteq\typ$ or $\vv_1\subseteq\typ_1$, algebras $(\wpf F_{\vv}(X),\Om,\cup)$, $(\mathcal{P}^{<\om} F_{\vv}(X),\Om,\cup,\emptyset)$ or $(\wpf F_{\vv_1}(X),\Om,\cup,\{1\})$ have the
universality property for semilattice ordered algebras in
$\mm_{\vv}$, $\mm_{\vv}^0$ or $\mm_{\vv_1}$, respectively, but in general, the algebras themself doesn't have to
belong to these varieties. 

\begin{exm}\label{sec4:exm1} 
Let $\vv$ be a variety of semilattices $(A,\cdot)$ and $\vv_1$ be a variety of semilattices $(A,\cdot,1)$ with the greatest element $1$.

Consider the free semilattice $(F_{\vv}(X),\cdot)$ over a set $X$
in the variety $\vv$,  the free algebra $(F_{\vv_1}(X),\cdot,1)$ over a set $X$
in the variety $\vv_1$ and their two generators $x,y\in X$. One can
easily see that
\begin{equation*}
\{x,y\}\cdot\{x,y\}=\{x,x\cdot y,y\}\neq\{x,y\}.
\end{equation*}
This shows that the algebra $(\wpf F_{\vv}(X),\cdot,\cup)$ does not
belong to the variety $\mm_{\vv}$. This also immediately implies that algebras $(\mathcal{P}^{<\om} F_{\vv}(X),\cdot,\cup,\emptyset)$
and $(\wpf F_{\vv}(X) F_{\vv}(X),\cdot,\cup,\{1\})$ do not belong to varieties $\mm_{\vv}^0$ and $\mm_{\vv_1}$, respectively.
\end{exm}

\begin{cor}\cite{PZ12a}\label{sec4:cor1}
The semilattice ordered algebra
$(\wpf F_{\vv}(X),\Om,\cup)$ is free over a set $X$ in the variety
$\mm_{\vv}$ if and only if $(\wpf F_{\vv}(X),\Om,\cup)\in\mm_{\vv}$.
\end{cor}

\begin{cor}
The semilattice ordered algebra
$(\mathcal{P}^{<\om} F_{\vv}(X),\Om,\cup,\emptyset)$ is free over a set $X$ in the variety
$\mm_{\vv}^0$ if and only if $(\mathcal{P}^{<\om} F_{\vv}(X),\Om,\cup,\emptyset)\in\mm_{\vv}^0$.
\end{cor}

\begin{cor}
The semilattice ordered algebra
$(\wpf F_{\vv_1}(X),\Om,\cup,\{1\})$ is free over a set $X$ in the variety
$\mm_{\vv_1}$ if and only if $(\wpf F_{\vv_1}(X),\Om,\cup,\{1\})\in\mm_{\vv_1}$.
\end{cor}

\begin{cor}
The semilattice ordered algebra
$(\mathcal{P}^{<\om} F_{\vv_1}(X),\Om,\cup,\emptyset,\{1\})$ is free over a set $X$ in the variety
$\mm_{\vv_1}^0$ if and only if $(\mathcal{P}^{<\om} F_{\vv_1}(X),\Om,\cup,\emptyset,\{1\})\in\mm_{\vv_1}^0$.
\end{cor}

\begin{cor}\label{sec4:thm2}\cite{PZ12a}
Let $(F_{\typ}(X),\Om)$ be the free algebra over a set $X$ in the variety $\typ$.
The extended power algebra $(\wpf F_{\typ}(X),\Om,\cup)$ is free
over $X$ in the variety $\mm_{\typ}$ of all semilattice
ordered $\typ$-algebras.
\end{cor}
Note that, by Corollary \ref{gl3}, the
same holds also for any variety defined by a set of linear identities.

\begin{thm}\cite{PZ12a}\label{sec4:thm3}
Let $\vv$ be a variety defined by a set of linear identities. The extended power algebra $(\wpf
F_{\vv}(X),\Om,\cup)$ is free over $X$ in the variety
$\mm_{\vv}$ of all semilattice ordered $\vv$-algebras.
\end{thm}

\begin{thm}\label{sec4:thm4}
Let $\vv$ be a variety defined by a set of linear regular identities. The $\emptyset$-extended power algebra $(\mathcal{P}^{<\om} F_{\vv}(X),\Om,\cup,\emptyset)$ is free over $X$ in the variety
$\mm_{\vv}^0$ of all $0$-semilattice ordered $\vv$-algebras.
\end{thm}
For a variety $\vv$ let $\vv^{*}$ be its {\bf\emph{linearization}}, the
variety defined by all linear identities satisfied in $\vv$.
Obviously, $\vv^{*}$ contains $\vv$ as a subvariety.

Since by Theorem \ref{gl} and Corollary \ref{gl1} for any subvariety
$\vv\subseteq\mho$ the algebra $(\wpf F_{\vv}(X),\Om)$ satisfies only those identities
resulting through identification of variables from the linear identities true in $\vv$, then for each subvariety $\vv\subseteq\mho$, the algebra $(\wpf F_{\vv}(X),\Om)$ belongs to $\vv^*$,
but it does not belong to any its proper subvariety.

\begin{cor}\label{sec4:cor2}
Let $X$ be an infinite set. For any subvariety $\vv\subseteq\typ$, we have
\begin{equation*}
\mm_{\vv^{*}}={\rm HSP}((\wpf F_{\vv^{*}}(X),\Om,\cup)).
\end{equation*}
\end{cor}

Let $\mm$ be a non-trivial subvariety of
$\mm_{\vv}$ and $X$ be a set. By \cite[Chapter 3.3]{RS02} the
congruence
$$\Phi_{\mm}(X):=\bigcap\{\phi\in Con(\wpf F_{\vv}(X),\Om,\cup)\mid
(\wpf F_{\vv}(X)/{\phi},\Om,\cup)\in \mm\}$$
is the
$\mm$-{\bf\emph{replica congruence}} of $(\wpf F_{\vv}(X),\Om,\cup)$ and
$(\wpf F_{\vv}(X)/{\Phi_{\mm}(X)}, \Om,\cup)$ is called the
$\mm$-{\bf\emph{replica}} of $(\wpf F_{\vv}(X),\Om,\cup)$.

Let $(B,\Om,+)\in \mm$. By the universality property of replication
(see \cite[Lemma 3.3.1.]{RS02}), for each homomorphism
$\overline{\overline{h}}\colon(\wpf F_{\vv}(X),\Om,\cup)\to
(B,\Om,+)$, there is a unique homomorphism $$\widehat{h}\colon(\wpf
F_{\vv}(X)/{\Phi_{\mm}(X)},\Om,\cup)\to (B,\Om,+)$$ such that
\begin{equation*}
\overline{\overline{h}}=\widehat{h}\circ nat \Phi_{\mm}(X),
\end{equation*}
where $nat\Phi_{\mm}(X)$ is the natural projection onto the quotient
$\wpf F_{\vv}(X)/{\Phi_{\mm}(X)}$. Hence, the universality property for $(\wpf
F_{\vv}(X),\Om,\cup)$ yields the following commuting diagram for any
mapping $h\colon X\to B$: \begin{center}
\setlength{\unitlength}{1mm}
\begin{picture}(80,25)
\put(0,17){\makebox(0,0){$X$}} \put(25,17){\makebox(0,0){$(\wpf
F_{\vv}(X),\Om,\cup)$}} \put(82,17){\makebox(0,0){$(\wpf
F_{\vv}(X)/{\Phi_{\mm}(X)},\Om,\cup)$}}
\put(25,2){\makebox(0,0){$(B,\Om,+)$}}
\put(5,17){\makebox(0,0){$\hookrightarrow$}}
\put(40,17){\vector(1,0){20}} \put(25,14){\vector(0,-1){7}}
\put(2,14){\vector(1,-1){8}} \put(60,14){\vector(-2,-1){20}}
\put(10,10){\makebox(0,0){$h$}}
\put(40,10){\makebox(0,0){$\widehat{h}$}}
\put(20,10){\makebox(0,0){$\overline{\overline{h}}$}}
\put(5,20){\makebox(0,0){$i$}}
\put(50,20){\makebox(0,0){$nat\Phi_{\mm}(X)$}}
\end{picture}
\end{center}
As a result, we obtain

\begin{thm}\label{sec4:thm5} The $\mm$-replica of the algebra
$(\wpf F_{\vv}(X),\Om,\cup)$ is free over a set $X$ in the variety
$\mm\subseteq \mm_{\vv}$.
\end{thm}

\begin{cor}\label{sec4:cor3}
Let $(\wpf F_{\vv}(X),\Om,\cup)\in \mm$. Then it is free in
$\mm\subseteq \mm_{\vv}$ over a set $X$.
\end{cor}

Additionally, by Theorem \cite[Theorem 3.9]{PZ11} one can obtain a characterization of free algebras in
the quasivariety $\typ_\mm$ of $\Om$-subreducts of semilattice ordered algebras in a given variety $\mm$.

\begin{cor}\label{L:free01}
The free algebra $(F_{\typ_\mm}(X),\Om)$ over $X$ in $\typ_\mm$ is
isomorphic to the full $\Om$-subreduct $\langle X\rangle_{\Om}$ of
the free semilattice ordered algebra $(F_{\mm}(X),\Om,+)$ in $\mm$.
\end{cor}

Free algebras in a quasivariety $\typ_\Om$ are also free in
the variety $V(\typ_\Om)$ generated by $\typ_\Om$ (see
\cite{M73}). But note that even if we have a free semilattice ordered algebra
$(F_{\mm}(X),\Om,+)$ in a given quasivariety
$\mm\subseteq\mathcal{SV}$, its full $\Om$-subreduct $(\langle
X\rangle_{\Om},\Om)$ needn't be a free algebra in $\vv$.
\section{Applications}\label{sec:cdi}
\subsection{Idempotent SLO algebras}\label{sec:idem}
As we have already shown in Section \ref{sec:free}, identities true
in varieties of semilattice ordered algebras are determined by appropriate
extended power algebras of algebras and their homomorphic images. Entropic and symmetric identities are both linear and regular. In this section we focus on
the idempotent identities which are linear only if the operation $\om$ occuring there is unary.

Extended power algebras are very rarely idempotent. Note that if the
power algebra $(\mathcal{P}_{>0}A,\Om)$ of
$(A,\Om)$ is idempotent then the algebra $(A,\Om)$ must be
idempotent too. Further, if $(A,\Om)$ is idempotent then for any
non-empty subset $B$ of $A$ and $\om\in\Om$, we have $B\subseteq
\om(B,\ldots, B)$. Moreover, as an easy consequence of results of A.
Romanowska and J.D.H. Smith \cite[Proposition 2.1]{RS89} for an
idempotent algebra $(A,\Om)$, a non-empty subset
$B\in\mathcal{P}_{>0}A$ is a subalgebra of $(A,\Om)$ if and only if
$\om(B,\ldots,B)=B$ for each $\om\in \Om$.

\begin{cor}\cite{PZ15}\label{cor2}
The power algebra $(\mathcal{P}_{>0}A,\Om)$ of an idempotent algebra
$(A,\Om)$ is idempotent if and only if each non-empty subset $B$ of
$A$ is a subalgebra of $(A,\Om)$.
\end{cor}

\begin{exm}\cite{PZ15}
An algebra $(A,\Om)$ such that $\om(a_1,\ldots,a_n)\in
\{a_1,\ldots,a_n\}$, for each $n$-ary $\om\in\Om$ and
$a_1,\ldots,a_n\in A$, is called {\bf\emph{conservative}}. By Corollary
\ref{cor2}, the power algebra of any conservative algebra is always
idempotent. In particular, the power algebra of a chain, the power
algebra of a left zero-semigroup \cite{RS02}, the power algebra of
an equivalence algebra \cite{JM01} or the power algebra of a
tournament \cite{JMMM98} are all idempotent.
\end{exm}

Let $\theta$ be a congruence on an idempotent algebra $(A,\Om)$.
Obviously, $a\;\theta\; \om(a,\ldots,a)$ for each $a\in A$ and
$\om\in\Om$. On the other hand, it is not always true that
$X\;\theta\;\om(X,\ldots,X)$ for a subset $X$ of $A$, if
$(\mathcal{P}_{>0}A,\Om)$ is not idempotent. It is enough to
consider the equality relation on $(A,\Om)$ in such case.

Let $(M,\Om)$ be an idempotent and entropic algebra. Denote by $\mathcal{I}$ the variety of all
idempotent $\tau$-algebras of type
$\tau:\Om\sucdot\;\{\cup\}\rightarrow\mathbb N^+$.
Then
$Con_{\mathcal{I}}(\mathcal{P}_{>0}^{<\om}M)$ is the set of all congruence
relations $\gamma$ on $(\mathcal{P}_{>0}^{<\om}M,\Om,\cup)$, such that the
quotient $(\mathcal{P}_{>0}^{<\om}M^{\gamma},\Om)$ is idempotent. By
\cite[Section 1.4.3]{RS02} $Con_{\mathcal{I}}(\mathcal{P}_{>0}^{<\om}M)$ is
an algebraic subset of the lattice of all congruences of
$(\mathcal{P}_{>0}^{<\om}M,\Om,\cup)$. Recall that the least element in
$(Con_{\mathcal{I}}(\mathcal{P}_{>0}^{<\om}M),\subseteq)$ is
the $\mathcal{I}$-replica congruence of
$(\mathcal{P}_{>0}^{<\om}M,\Om,\cup)$.

Let $(M,\Om,1)$ be an idempotent and entropic algebra with the unit $1\in M$ and let the algebra
$(\mathcal{P}^{<\om}M,\Om,\cup,\emptyset,\{1\})$ be the $\emptyset$-extended power algebra with the unit $\{1\}$.
Let us define a binary relation $\rho$ on the set
$\mathcal{P}^{<\om}M$ in the following way:
\begin{eqnarray}\label{eqn3} A\; \rho \; B &\Leftrightarrow
&\text{there exist a $k$-ary
term $t$ and an $m$-ary term $s$}\\ &&\text{both of type $\Om$ such that}\nonumber\\
&&A\subseteq t(B,B,\ldots,B)\;\;\; {\rm and} \; \; \; B\subseteq
s(A,A,\ldots,A).\nonumber
\end{eqnarray}

It was proved in \cite{PZ12b} that the relation $\rho|_{\mathcal{P}_{<0}^{<\om}M}$ is the
$\mathcal{I}$-replica congruence of $(\mathcal{P}^{<\om}_{>0}M,\Om,\cup)$ and is equal to the relation:
\begin{eqnarray}\label{eqn4}
A\; \alpha \; B \;\; &\Leftrightarrow \;\; \langle A \rangle =
\langle B \rangle,
\end{eqnarray}
where $\langle A \rangle$ is the subalgebra of $(M,\Om)$ generated by the set $A$.
Therefore, $(\mathcal{P}^{<\om}_{>0}M/\rho,\Om,\cup)\cong(\{\langle A \rangle\colon A\in \mathcal{P}^{<\om}_{>0}M\},\Om,+)$, where
for each $n$-ary complex operation $\om\in \Om$ and non-empty subsets
$A_1,\ldots,A_n$ of $M$
\begin{equation}\label{gen}
\om(\langle A_1\rangle,\ldots,\langle A_n\rangle) =\langle\om(A_1,\ldots,A_n)\rangle ,\; {\rm and}
\end{equation}
\begin{equation}\label{eq1}\langle A_1 \rangle+\langle A_2 \rangle=\langle A_1\cup A_2\rangle.
\end{equation}

It is easy to notice that
\[
\emptyset\; \rho \; A \quad\Leftrightarrow\quad A=\emptyset
\]
and
\[
\{1\}\; \rho \; A \quad\Leftrightarrow\quad A=\{1\}.
\]
Hence, $\rho$ is also the $\mathcal{I}$-replica congruence of $(\mathcal{P}^{<\om}M,\Om,\cup)$ and
$(\mathcal{P}^{<\om}M/\rho,\Om,\cup)\cong(\{\langle A \rangle\colon A\in \mathcal{P}^{<\om}M\},\Om,+)$,
with assumption that $\langle \emptyset\rangle =\emptyset$.

By Theorem \ref{sec4:thm5} we have the following
\begin{thm}
Let $\mathcal{M}$ be the variety of all idempotent and entropic $\Om$-algebras $(M,\Om)$. The $0$-semilattice ordered algebra
$(\{\langle A \rangle\colon A\in \mathcal{P}^{<\om}F_{\mathcal{M}}(X)\},\Om,+,\emptyset)$
is free over a set $X$ in the variety $\mm_{\mathcal{M}}^0$.
\end{thm}

Moreover, for any $k$-ary term $t$ and a subset $1\notin S\subset M$
\[
\{1\}\cup S\subseteq t(S,\ldots,S)\quad\Leftrightarrow\quad \exists(s_1,\ldots,s_k\in S)\quad 1=t(s_1,\ldots,s_k).
\]
This shows that for a subset $\emptyset\neq S\subset M$ such that $1\notin S$
\begin{eqnarray*}
\{1\}\cup S\; \rho \; S &\Leftrightarrow
&\text{there exist a $k$-ary
term $t$ of type $\Om$ and $s_1,\ldots,s_k\in S$}\\
&&\text{such that} \; 1=t(s_1,\ldots,s_k).
\end{eqnarray*}

\begin{lm}
Let $(M,\Om,1)$ be an idempotent and entropic algebra with the unit $1\in M$. Let us assume that the algebra $(M,\Om,1)$ satisfies
the following condition:
\begin{align}\label{con:1}
\forall(\om\in\Om)\; \forall(x_1,\ldots,x_n\in M)\quad \om(x_1,\ldots,x_n)=1\quad \Rightarrow\quad \forall(1\leq i\leq n)\quad x_i=1.
\end{align}

Then  $\rho$ is the $\mathcal{I}$-replica congruence of $(\mathcal{P}^{<\om}_{>0}M,\Om,\cup,\{1\})$ and
\[
(\mathcal{P}^{<\om}_{>0}M/\rho,\Om,\cup,\{1\})\cong(\{\langle A \rangle\colon 1\notin A\in \mathcal{P}^{<\om}_{>0}M\}\cup\{\langle A\cup\{1\} \rangle\colon A\in \mathcal{P}_{>0}^{<\om}M\}\cup\{1\},\Om,+,\{1\}).
\]
\end{lm}

\begin{thm}
Let $\mathcal{M}_1$ be the variety of all idempotent and entropic $\Om$-algebras $(M,\Om,1)$ with the unit $1$ which satisfy Condition \eqref{con:1}.

Then the semilattice ordered algebra
$(\{\langle A \rangle\colon 1\notin A\in \mathcal{P}^{<\om}_{>0}F_{\mathcal{M}}(X)\}\cup\{\langle A\cup\{1\} \rangle\colon 1\notin A\in \mathcal{P}_{>0}^{<\om}F_{\mathcal{M}}(X)\}\cup\{1\},\Om,+,\{1\})$
is free over a set $X$ in the variety $\mm_{\mathcal{M}_1}$.
\end{thm}

\begin{cor}\label{cor:01}
Let $\mathcal{M}_1$ be the variety of all idempotent and entropic $\Om$-algebras $(M,\Om,1)$ with the unit $1$ which satisfy Condition \eqref{con:1}.

Then the $0$-semilattice ordered algebra
$(\{\langle A \rangle\colon 1\notin A\in \mathcal{P}^{<\om}F_{\mathcal{M}}(X)\}\cup\{\langle A\cup\{1\} \rangle\colon 1\notin A\in \mathcal{P}^{<\om}F_{\mathcal{M}}(X)\},\Om,+,\emptyset,\{1\})$
is free over a set $X$ in the variety $\mm^0_{\mathcal{M}_1}$.
\end{cor}

\subsection{Commutative double idempotent semirings.}
\begin{de}
A {\bf\emph{semiring}} is an algebra $(S,\cdot,+)$ such that
\begin{enumerate}
\item $(S,\cdot)$ is a semigroup,
\item $(S,+)$ is a commutative semigroup, 
\item for $a,b,c\in S$, $a\cdot(b+c)=a\cdot b+a\cdot c$ and $(b+c)\cdot a=b\cdot a+c\cdot a$.
\end{enumerate}
\end{de}
A semiring is said to be {\bf\emph{commutative}} if the semigroup $(S,\cdot)$ is commutative.
A semiring is {\bf\emph{additively}} [{\bf\emph{multiplicatively}}, respectively]  {\bf\emph{idempotent}} if the semigroup
$(S,+)$ [$(S,\cdot)$, respectively] is idempotent. Hence, additively idempotent semirings are simply semilattice ordered semigroups.
\begin{remark}
Notice that in the literature of semirings there
are several definitions depending on whether the algebra contains an identity
and/or a zero element. See e.g. \cite{G99}.
\end{remark}
Let $\mathcal{SG}$ denote the variety of all semigroups. Since associativity is a linear identity, then by Theorem \ref{sec4:thm3} the extended power algebra $(\wpf F_{\mathcal{SG}}(X),\cdot,\cup)$ is free over $X$ in the variety
$\mm_{\mathcal{SG}}$ of all additively idempotent semirings. The similar description of free additively idempotent semirings, where $(S,\cdot)$ belongs to a subvariety of $\mathcal{SG}$, defined by a set of linear identities, is also true.

On the other hand, idempotency is not a linear identity so the extended power algebra of the free algebra in the variety of all idempotent semigroups (\emph{bands}) need not be idempotent. In consequence, such algebra is not free algebra in the variety of all additively and multiplicatively idempotent semirings. Double idempotent semirings were called \emph{distributive $\cdot$-bisemilattices} by R. McKenzie and A. Romanowska and studied in \cite{MR79}.

If the semigroup $(S,\cdot)$ is also entropic (\emph{normal band}), i.e. it satisfies for $a,b,c,d\in S$
\[
a\cdot b\cdot c\cdot d=a\cdot c\cdot b\cdot d,
\]
then by Theorem \ref{sec4:thm3} the algebra  $(\{\langle A \rangle\colon A\in \mathcal{P}^{<\om}_{>0} F_{\mathcal{NB}}(X)\},\cdot,+)$ of all finitely generated subalgebras of free algebra  $F_{\mathcal{NB}}(X)$ in the variety  $\mathcal{NB}$ of all normal bands, is free in the variety of double idempotent semirings with entropic multiplication reduct. This result coincides with a construction given by Zhao in \cite{Z02} where he applied so called \emph{closed subsets}, since each non-empty subset of a normal band is closed if and only if it is a subband.

Quite recently Chajda and Langer \cite{ChL18} investigated commutative double idempotent semirings $(S,\cdot,+,0,1)$ with two constants $0$ and $1$, such that $(S,+,0)$ and $(S,\cdot,1)$ are semilattices with the least element $0$ and the greatest element $1$, respectively, and for each $x\in S$,
\[
x\cdot 0=0\cdot x=0.
\]
Clearly, such semirings are exactly $0$-semilattice ordered semilattices with a unit $1$.
In particular, Chajda and Langer described free algebras in the variety $\mathcal{Z}$ of all commutative double idempotent semirings with two constants. Since commutative semigroups are trivially entropic some results in \cite{ChL18} immediately follows by general ones.

Let $\mathcal{SL}_1$ be the variety of all semilattices with a unit $1$ and let $(F_{\mathcal{SL}}(X),\cdot)$ be the free semilattice in
$\mathcal{SL}$ generated by a set $X$. Obviously, Condition \eqref{con:1} is satisfied in any idempotent monoid, so by Corollary \ref{cor:01} we obtain:
\begin{thm}\label{thm:cdis}
The $0$-semilattice ordered algebra
$(\{\langle A \rangle\colon 1\notin A\in \mathcal{P}^{<\om}F_{\mathcal{SL}}(X)\}\cup\{\langle A \rangle\cup\{1\}\colon 1\notin A\in \mathcal{P}^{<\om}F_{\mathcal{SL}}(X)\},\cdot,+,\emptyset,\{1\})$
is free over a set $X$ in the variety $\mathcal{Z}=\mm^0_{\mathcal{SL}_1}$.
\end{thm}

Therefore, directly by Disjunctive Form Lemma \ref{slowa} every term $t(x_1,\ldots,x_n)\in F_{\mm^0_{\mathcal{SL}_1}}(X)$ is a sum of some products of variables
$x_1,\ldots,x_n$ (see \cite[Lemma 4]{ChL18}).

Further, it is well known that free algebra generated by $X$ in the variety $\mathcal{SL}_0$ is isomorphic to the semilattice $(\mathcal{P}X,\cup)$ of all subsets of $X$. Then the number of different $n$-ary terms in $F_{\mm^0_{\mathcal{SL}_1}}(X)$ is less or equal to $2^{2^n}$ (see \cite[Corollary 5]{ChL18}). Locally finitness of $\mathcal{Z}=F_{\mm^0_{\mathcal{SL}_1}}(X)$ follows also by Theorem \ref{thm13}.

If $X$ is a finite set, then by Theorem \ref{thm:cdis} the cardinality of $F_{\mm^0_{\mathcal{SL}_1}}(X)$ is equal double cardinality of the set of all subalgebras of the free algebra in the variety $\mathcal{SL}$ including the empty set:
\[
|F_{\mm^0_{\mathcal{SL}_1}}(X)|=2|\{(A,\cdot)\colon (A,\cdot)\leq F_{\mathcal{SL}}(X)\}|.
\]
In particular, for $X=\emptyset$ there are only one subalgebra of $F_{\mathcal{SL}}(\emptyset)$: the empty set. Then $F_{\mm^0_{\mathcal{SL}_1}}(\emptyset)\cong (\{\emptyset,\{1\}\},\cdot,\cup,\emptyset,\{1\})$. Further, for $X=\{x\}$ we obtain
\[
F_{\mm^0_{\mathcal{SL}_1}}(\{x\})\cong (\{\emptyset,\{x\},\{1\},\{x,1\}\},\cdot,\cup,\emptyset,\{1\}).
\]
For $X=\{x,y\}$, the free semilattice $F_{\mathcal{SL}}(X)$ on two generators has three elements: $x,y,xy$ and 7 subalgebras (including the empty set): $\emptyset,\{x\}, \{y\}, \{xy\}, \{x,xy\}, \{y,xy\}, \{x,y,xy\}$. Hence $|F_{\mm^0_{\mathcal{SL}_1}}(\{x,y\})|=14$.

Referring to the notion introduced in \cite{ChL18} we say that a subset $A$ of $F_{\mathcal{SL}}(X)$ is {\bf\emph{reduced}} if
\begin{align*}
&\forall(a\in A)\; \forall(k\in \mathbb{N}^+)\; \forall(b_1,\ldots,b_k\in A\setminus\{a\})\quad a\neq b_1\cdots b_k.
\end{align*}

It is evident that for each finitely generated subalgebra $(C,\cdot)$ of $F_{\mathcal{SL}}(X)$ there exists exactly one finite reduced subset $A_r\subseteq F_{\mathcal{SL}}(X)$ such that $(C,\cdot)=\langle A_r \rangle$. Hence, the cardinality of the free algebra $F_{\mm_{\mathcal{SL}}}(X)$ in the variety $\mm_{\mathcal{SL}}$ of all semilattice ordered semilattices is equal to cardinality of all reduced subsets of $F_{\mathcal{SL}}(X)$. This implies that the cardinality of $F_{\mm^0_{\mathcal{SL}_1}}(X)$ is equal double cardinality of the set of all reduced subsets of $F_{\mathcal{SL}}(X)$ including empty set.

\subsubsection*{Acknowledgements}
While working on this paper, the authors were supported by the
Grant of Warsaw University of Technology 504/04259/1120.

\end{document}